\documentclass{daj}
\dajAUTHORdetails{%
	title = {The structure of multiplicative functions with small partial sums}, 
	author = {Dimitris Koukoulopoulos and K. Soundararajan},
	plaintextauthor = {Dimitris Koukoulopoulos and K. Soundararajan},
	keywords = {multiplicative function, partial sums, partial sums over primes, LSD method, converse theorem},
}   

\dajEDITORdetails{%
	year={2020},
	number={6},
	received={22 October 2019},   
	published={12 May 2020},  
	doi={10.19086/da.11963},       
}   

\usepackage{amssymb,amsmath,amsthm}
\usepackage{mathrsfs} 
\usepackage{bbm} 
\usepackage{enumitem}

\numberwithin{equation}{section}

\newtheorem{mainthm}{Theorem}
\newtheorem{thm}{Theorem}[section]
\newtheorem{lem}[thm]{Lemma}
\newtheorem{prop}[thm]{Proposition}

\theoremstyle{definition}
\newtheorem{dfn}[thm]{Definition}

\theoremstyle{remark}

\newtheorem*{remark*}{Remark}

\newtheoremstyle{claim} 
    {1em}                    
    {1em}                    
    {}                   
    {}                           
    {\bfseries}                   
    {.}                          
    {.5em}                       
    {}  
    
\theoremstyle{claim}

\newcommand{\N}{\mathbb{N}}
\newcommand{\Z}{\mathbb{Z}}

\newcommand{\R}{\mathbb{R}}
\newcommand{\C}{\mathbb{C}}

\newcommand{\CA}{\mathcal{A}}

\newcommand{\CF}{\mathcal{F}}

\newcommand{\CT}{\mathcal{T}}

\newcommand{\bs}\boldsymbol{}

\newcommand{\eq}[2]{ \begin{equation} \label{#1}\begin{split} #2 \end{split} \end{equation} }

\newcommand{\als}[1]{\begin{align*} #1 \end{align*} }

\newcommand{\dee}{\mathrm{d}}

\renewcommand{\tilde}{\widetilde}
\newcommand{\eps}{\varepsilon}
\renewcommand{\phi}{\varphi}

\renewcommand{\Re}{{\rm Re}}

\renewcommand{\bar}[1]{\overline{#1}}

\begin{document}

\begin{frontmatter}[classification=text]
	
	\title{The structure of multiplicative functions with small partial sums} 
	\author[dimitris]{Dimitris Koukoulopoulos\thanks{Supported by the Natural Sciences and Engineering Research Council of Canada (Discovery Grants 435272-2013 and 2018-05699) and by the Fonds de Recherche du Qu\'ebec -- Nature et Technologies (\'Etablissement de nouveaux chercheurs universitaires 2016-NC-189765; Projet de recherche en \'equipe 2019-PR-256442).}}
	\author[sound]{K. Soundararajan\thanks{Supported by the National Science Foundation Grant and through a Simons Investigator Grant from the Simons Foundation. In addition, the paper was completed while the second author was a senior Fellow at the ETH Institute for Theoretical Studies, whom he thanks for their warm and generous hospitality.}}

\begin{abstract}
The Landau-Selberg-Delange method provides an asymptotic formula for the partial sums of a multiplicative function whose average value on primes is a fixed complex number $v$. The shape of this asymptotic implies that $f$ can get very small on average only if $v=0,-1,-2,\dots$. Moreover, if $v<0$, then the Dirichlet series associated to $f$ must have a zero of multiplicity $-v$ at $s=1$. In this paper, we prove a converse result that shows that if $f$ is a multiplicative function that is bounded by a suitable divisor function, and $f$ has very small partial sums, then there must be finitely many real numbers $\gamma_1$, $\dots$, $\gamma_m$ such that $f(p)\approx -p^{i\gamma_1}-\cdots-p^{-i\gamma_m}$ on average. The numbers $\gamma_j$ correspond to ordinates of zeros of the Dirichlet series associated with  $f$, counted with multiplicity. This generalizes a result of the first author, who handled the case when $|f|\le 1$ in previous work.
\end{abstract}
\end{frontmatter}


\section{Introduction}

Throughout this paper $f$ will denote a multiplicative function and 
\[
L(s,f) =\sum_{n=1}^{\infty} \frac{f(n)}{n^s}  
\]
will be its associated Dirichlet series, which is assumed to converge absolutely in $\Re(s)>1$.  We then have
\[
-\frac{L^{\prime}}{L} (s,f) =\sum_{n=1}^{\infty} \frac{\Lambda_f(n)}{n^s}, 
\]
for certain coefficients $\Lambda_f(n)$ that are zero unless $n$ is a prime power.  

Let $D$ denote a  fixed positive integer.  We shall restrict attention to the class of multiplicative functions $f$ such that 
\eq{Lambda_f}{
|\Lambda_f (n) |\le D\cdot \Lambda(n) 
}
for all $n$.   This is a rich class of functions that includes most of the important multiplicative functions that arise in number theory.  
For example, the M{\" o}bius function, the Liouville function, divisor functions, and coefficients of automorphic forms (or if one prefers an axiomatic approach, $L$-functions in the Selberg class) satisfying a Ramanujan bound are all covered by this framework.

When $f(p)\approx v$ in an appropriately strong form, Selberg \cite{selberg} built on ideas of Landau \cite{landau,landau2} to prove that
\eq{Selberg}{
\sum_{n\le x} f(n) = \frac{c(f,v)}{\Gamma(v)} x(\log x)^{v-1} + O_f(x(\log x)^{v-2}) ,
}
where $c(f,v)$ is some a non-zero constant given in terms of an Euler product. Delange \cite{delange} strengthened this theorem to a full asymptotic expansion:  
\[
\sum_{n\le x} f(n) = x(\log x)^{v-1} \sum_{j=0}^{J-1} \frac{c_j(f,v)}{\Gamma(v-j)(\log x)^j}  
	+ O(x(\log x)^{v-1-J})  
\]
for any $J\in\N$, where $c_0(f,v)=c(f,v)$. 
In particular,  if $v\in\{0,-1,-2,\dots\}$, then the partial sums of $f$ satisfy the bound
\eq{f-small}{
\sum_{n\le x} f(n) \ll \frac{x}{(\log x)^A} \qquad(x\ge 2) 
}
for any $A>0$. 

This paper is concerned with the converse problem:  assuming that \eqref{f-small} holds for some $A>D+1$, what can be deduced about $f(p)$? If we already knew that $f(p)\approx v$ on average, then relation \eqref{Lambda_f} would imply that $|v|\le D$. Comparing \eqref{f-small} with \eqref{Selberg}, we conclude that $v\in\{0,-1,-2,\dots,-D\}$. The goal of this paper is to prove such a converse result to the Landau-Selberg-Delange theorem without assuming prior knowledge of the average behavior of $f(p)$.  

This problem was studied in the case $D=1$ by the first author \cite{dk-gafa}.  If \eqref{f-small} holds for some $A>2$, then by partial summation one can see that $L(s,f)$ converges (conditionally) on the line $\Re(s)=1$.   The work in \cite{dk-gafa} established that on the line $\Re(s)=1$ the function $L(s,f)$ can have at most one simple zero.  If $L(1+it, f) \neq 0$ for all $t$, then 
\[
\lim_{x\to\infty} \frac{1}{\pi(x)} \sum_{p\le x} f(p)=0,
\]
while if $L(1+i\gamma,f)=0$ for some (unique) $\gamma\in\R$ then 
\[
\lim_{x\to\infty} \frac{1}{\pi(x)} \sum_{p\le x} (f(p) + p^{i\gamma}) =0 .
\]
In this paper we establish a generalization of this result for larger values of $D$.

\begin{mainthm}\label{converse-thm}
Fix a natural number $D$ and a real number $A>D+2$.  Let $f$ 
be a multiplicative function such that $|\Lambda_f|\le D\cdot\Lambda$, and such that 
\[
\sum_{n\le x}f(n) \ll \frac{x}{(\log x)^A}
\]
for all $x\ge2$. Then there is a unique multiset $\Gamma$ of at most $D$ real numbers such that 
\[
\Big| \sum_{p\le x} \Big( f(p) + \sum_{\substack{ \gamma \in \Gamma \\ |\gamma| \le T}} p^{i\gamma}\Big) \log p \Big| \le C_1 \frac{x}{\sqrt{\log x}}  + C_2 \frac{x}{\sqrt{T}}
\]
for all $x,T\ge2$, where $C_1 = C_1(f,T)$ is a constant depending only on $f$ and $T$, and $C_2$ is an absolute constant.  
In particular,  
\[
\lim_{x\to\infty} \frac{1}{x} \sum_{p\le x}\Big( f(p) + \sum_{\gamma \in \Gamma} p^{i\gamma}\Big) \log p = 0 .
\]
\end{mainthm}

 The multiset $\Gamma$ consists of the ordinates of the zeros of $L(s,f)$ on the line $\Re(s)=1$, repeated according to their multiplicity. Its rigorous construction is described in Proposition \ref{Prop1}. The constant $C_1=C_1(f,T)$ in Theorem \ref{converse-thm} can be calculated explicitly in terms of upper bounds for the Dirichlet series $L(s,f) \prod_{\gamma \in \Gamma, |\gamma|\le T} \zeta(s-i\gamma)$ and its derivatives, together with a lower bound for this quantity on the line segment $[1-iT, 1+iT]$. 
 
Qualitatively Theorem \ref{converse-thm} establishes the kind of converse theorem that we seek. There are two deficiencies in the theorem: first, the range $A > D+2$ falls short of the optimal result $A > D+1$ (which in the case $D=1$ was attained in \cite{dk-gafa}, and which we can attain in a 
special case -- see Section 5); and second, one would like an understanding of the uniformity with which the result holds.   On the other hand, the proof that we present is very simple, and we postpone the considerably more involved arguments needed for more precise versions of the theorem to another occasion.

\subsection*{Acknowledgements} We would like to thank Hongze Li for catching an error in an earlier version of the paper, and the referee of the paper for simplifying the proof of Lemma 4.2(a).


\section{The classes $\CF(D)$ and $\CF(D;A)$}


We introduce the classes of multiplicative functions ${\mathcal F}(D)$ and ${\mathcal F}(D;A)$, and establish some 
preliminary results.    Throughout $\tau_D$ will denote the $D$-th divisor function, which arises as the Dirichlet series coefficients of $\zeta(s)^D$.

\begin{dfn} \label{def1}  Given a natural number $D$, we denote by $\CF(D)$ the class of all multiplicative functions such that $|\Lambda_f|\le D\cdot\Lambda$.
\end{dfn}

\begin{lem}\label{lem:f-bound} Let $f$ be an element of ${\mathcal F}(D)$.  Then its inverse under Dirichlet convolution $g$ is also 
in ${\mathcal F}(D)$, and both $|f(n)|$ and $|g(n)|$ are bounded by $\tau_D(n)$ for all $n$. 
\end{lem}

\begin{proof} Note that 
\[
L(s,f) = \exp\bigg\{ \sum_{n=2}^\infty \frac{\Lambda_f(n)}{n^s\log n} \bigg\}
	= \sum_{j=0}^\infty \frac{1}{j!} \bigg( \sum_{n=1}^\infty \frac{\Lambda_f(n)}{n^s\log n} \bigg)^j, 
\]
so that by comparing coefficients  
\eq{f-Lambda_f}{
f(n) = \sum_{j=1}^\infty \frac{1}{j!} \sum_{n_1\cdots n_j=n} \frac{\Lambda_f(n_1)\cdots \Lambda_f(n_j)}{\log n_1\cdots \log n_j} .
}
Thus, by the definition of ${\mathcal F}(D)$, $|f(n)|$ is bounded by the coefficients of 
\[
\sum_{j=0}^{\infty} \frac 1{j!} \bigg( \sum_{n=1}^{\infty} \frac{D\Lambda(n)}{n^s \log n} \bigg)^{j} = \zeta(s)^D.
\] 
This shows that $|f(n)| \le \tau_D(n)$ for all $n$.  Since the inverse $g$ may be defined by setting $\Lambda_g(n)  = -\Lambda_f(n)$, 
it follows that $g$ is in ${\mathcal F}(D)$ and that $|g(n)| \le \tau_D(n)$ as well. 
\end{proof}

For later use, let us record that if $f\in {\mathcal F}(D)$, then for $\sigma >1$ we have 
\eq{2.2}
{\log L(s, f) = \sum_{n=2}^{\infty} \frac{\Lambda_f(n)}{n^s \log n} = \sum_{p} \frac{\Lambda_f(p)}{p^s \log n} + O(1)
= \sum_{p} \frac{f(p)}{p^s}+ O(1). 
}

We now introduce the class ${\mathcal F}(D;A)$, which is the subclass of multiplicative functions in ${\mathcal F}(D)$ 
with small partial sums.

\begin{dfn}  \label{def2}  Given a natural number $D$, and positive real numbers $A$ and $K$, we 
	denote by ${\mathcal F}(D;A,K)$ the class of functions $f\in \CF(D)$ such that
\[
	\Big| \sum_{n\le x} f(n) \Big| \le K \frac{x}{(\log x)^{A}}\qquad\text{for all}\quad x\ge 3.
\]
	The class ${\mathcal F}(D;A)$ consists of all functions lying in ${\mathcal F}(D;A,K)$ for some 
	constant $K$. 
\end{dfn}

The following proposition about the class $\CF(D;A)$ is an important stepping stone in the proof of Theorem \ref{converse-thm}.  
In particular, it gives a description of the multiset $\Gamma$ appearing in Theorem \ref{converse-thm}.  

\begin{prop} \label{Prop1}Suppose $f$ is in the class ${\mathcal F}(D;A)$ with $A >D+1$.  
\begin{enumerate}
	\item The series $L(s,f)$ and the series of derivatives $L^{(j)}(s,f)$ with $1\le j < A-1$ all converge uniformly in compact subsets of the region $\Re(s)\ge 1$.  
	
	\item  For any real number $\gamma$, there exists an integer $j\in[0,D]$ with $L^{(j)}(1+i\gamma, f) \neq 0$.  If $L(1+i\gamma,f)=0$, then $1+i\gamma$ is called a \emph{zero} of $L(s,f)$ and the \emph{multiplicity} of this zero is the smallest natural number $j$ with $L^{(j)}(1+i\gamma,f) \neq 0$.   
	
	\item  Counted with multiplicity, $L(s,f)$ has at most $D$ zeros on the line $\Re(s) =1$.  
	
	\item  Let $\Gamma$ denote the (possibly empty) multiset of ordinates $\gamma$ of zeros $1+i\gamma$ of $L(s,f)$, so that $\Gamma$ has cardinality at most $D$.  Let 
	${\widetilde \Gamma}$ denote a (multi-)subset of $\Gamma$, and let $\widetilde{m}$ denote the largest multiplicity of an element in $\widetilde{\Gamma}$.  
	The Dirichlet series 
	\[
	L(s,f_{\widetilde \Gamma}) = L(s,f) \prod_{\gamma \in {\widetilde \Gamma}} \zeta(s-i\gamma ) 
	\]
	and the series of derivatives $L^{(j)} (s,f_{\widetilde \Gamma})$ for $1 \le j <  A-{\widetilde m}-1$ all converge uniformly in compact subsets of the region $\Re(s)\ge 1$.
\end{enumerate}
\end{prop} 

We next establish the following lemma which contains part (a) of Proposition \ref{Prop1} and more.  The remaining parts of the proposition will 
be established in Section 4.   

\begin{lem} \label{lem:L-der-bound}
	Let $f \in {\mathcal F}(D;A,K)$ with $A >1$, and consider an integer $j\in[0,A-1)$. 
	For any $M\ge N \ge 3$ and any $s\in\C$ with $\Re(s)\ge1$, we have 
	\begin{equation}\label{L-der-bound1}
	\bigg| \sum_{N<n\le M} \frac{f(n)(\log n)^j}{n^s}\bigg|  \ll_{A,D} \frac{K(1+|s|)}{(\log N)^{A-j-1}}. 
	\end{equation}
	In particular, the series $L^{(j)}(s,f)$ converges uniformly in compact subsets of the region $\Re(s)\ge 1$.   
	Furthermore, it	satisfies the pointwise bound 
	\begin{equation}\label{L-der-bound2}
	|L^{(j)}(\sigma+ it, f)| \ll_{A,D} (K(1+|t|))^{\frac{D+j}{D+A-1}}
	\end{equation}
	for $s=\sigma+it$ with $\sigma\ge1$ and $t\in\R$.
\end{lem} 

\begin{proof} Since $|f|\le\tau_D$ by Lemma \ref{lem:f-bound}, all claims follow in the region $\Re(s)\ge2$ from the bound $\sum_{n>N}\tau_D(n)/n^2\ll_D(\log N)^{D-1}/N$. 
	
Let us now assume we are in the region $1\le \Re(s)\le2$. Using partial summation, we have
	\begin{align} 
	\sum_{N<n\le M} f(n) \frac{(\log n)^j}{n^{\sigma+it}}  
	&= \Big(\sum_{N<n\le M} f(n)\Big) \frac{(\log M)^j}{M^{\sigma+it}} - \int_N^M \Big( \sum_{N<n\le y} f(n) \Big) 
	\bigg( \frac{(\log y)^j}{y^{\sigma+ it}} \bigg)^{\prime} \dee y.
	\label{eqn:f-der-ps}
	\end{align} 
We estimate both terms on the right-hand side of \eqref{eqn:f-der-ps} using our assumption on the partial sums of $f$, thus obtaining that
	\begin{align*} 
	\sum_{N<n\le M} f(n) \frac{(\log n)^j}{n^{\sigma+it}}  
		&\ll \frac{K}{M^{\sigma-1}(\log M)^{A-j}}+ \int_N^M \frac{Ky}{(\log y)^A}\cdot  \frac{(\log y)^j (1+|t|)}{y^{\sigma +1}} \dee y \label{PS}\\
		&\ll \frac{K(1+|t|)}{(\log N)^{A-j-1}}. \nonumber
	\end{align*}
This establishes \eqref{L-der-bound1}. In particular, $L^{(j)}(s,f)$ converges uniformly in compact subsets of the region $1\le\Re(s)\le2$ by Cauchy's criterion.
	
To obtain \eqref{L-der-bound2}, we let $M\to\infty$ in \eqref{L-der-bound1} to find that 
	\begin{equation}\label{L-der-bound3}
	L^{(j)}(s,f) = \sum_{n\le N} \frac{f(n)(-\log n)^j}{n^s}+ O_{A,D}\bigg(\frac{K(1+|t|)}{(\log N)^{A-j-1}}\bigg) .
	\end{equation}
Since $|f(n)|\le \tau_D(n)$ by Lemma \ref{lem:f-bound}, the first term on the right side of \eqref{L-der-bound3} is bounded in size by 
	\[
	(\log N)^j \sum_{n\le N} \frac{\tau_D(n)}{n} \ll (\log N)^{D+j}. 
	\] 
Choosing $N = \exp((K(1+|t|)^{\frac{1}{D+A-1}})$ yields the desired bound. 
\end{proof}


\section{Two lemmas} 


Here we collect together a couple of disparate lemmas that will be used in the future.  Both of them are
of a standard nature, and proofs are provided for completeness. We begin with an asymptotic formula for partial sums of generalized divisor functions.  

\begin{lem} 
\label{lem:DivisorSums}
 Let ${\mathcal A} = \{ \alpha_1, \ldots, \alpha_m\}$ be a multiset, consisting of $k$ distinct elements, and 
 arranged so that $\alpha_1,\dots,\alpha_k$ denote these $k$ distinct values.  Suppose that 
these distinct values $\alpha_j$ appear in ${\mathcal A}$ with multiplicity $m_j$.   Let $\tau_{\mathcal A}(n)$ denote 
the multiplicative function 
\[
\tau_{\mathcal A}(n) = \sum_{d_1 \cdots d_m =n} d_1^{i\alpha_1} \cdots d_{m}^{i\alpha_m}. 
\]
Then for large $x$ we have
\[
\sum_{n\le x} \tau_{\mathcal A}(n) = \sum_{j=1}^{k} x^{1+i\alpha_j} P_{j,{\mathcal A}}(\log x) + O(x^{1-\delta}), 
\]
 where $P_{j, \CA}$ denotes a polynomial of degree $m_j-1$ with coefficients depending on $\CA$, and $\delta=\delta(\CA)$ is 
 some positive real number.
 \end{lem}   
 \begin{proof}   Note that in the region $\Re(s)> 1$ 
\[
 \sum_{n=1}^{\infty} \frac{\tau_{\mathcal A}(n)}{n^s}  = \prod_{j=1}^{k} \zeta(s-i\alpha_j)^{m_j}. 
\]
Now the lemma follows by a standard application of Perron's formula to write (with $c>1$)  
\[
\sum_{n\le x} \tau_{\mathcal A}(n) = \frac{1}{2\pi i} \int_{c-i\infty}^{c+i\infty} \prod_{j=1}^{k} \zeta(s-i\alpha_j)^{m_j}  \frac{x^s}{s} \dee s, 
\]
and then shifting contours to the left of the $1$-line and evaluating the residues of the poles of order $m_j$ at $1+i\alpha_j$.  
\end{proof}

Our second and final lemma gives a variant of the Brun--Titchmarsh theorem for primes in short intervals.   Define $\Lambda_j(n)$ by means of 
$$ 
(-1)^{j} \frac{\zeta^{(j)}(s)}{\zeta(s)} = \sum_{n} \frac{\Lambda_j(n)}{n^s}. 
$$ 
Thus $\Lambda_0(n)=1$ if $n=1$ and $0$ for $n >1$, while $\Lambda_1(n) = \Lambda(n)$ is the usual von Mangoldt function.   Using the identity
\begin{equation} 
\label{N3.1}
(-1)^{j+1} \frac{\zeta^{(j+1)}}{\zeta} = - \Big( (-1)^j \frac{\zeta^{(j)}}{\zeta}\Big)^{\prime} + \Big( -\frac{\zeta^{\prime}}{\zeta} \Big) 
\Big( (-1)^j \frac{\zeta^{(j)}}{\zeta}\Big),
\end{equation}  
 one can check easily that $\Lambda_j(n) \ge 0$ for all $j$ and $n$. In addition, $\Lambda_j(n)$ is supported on integers composed of at most $j$ distinct prime factors, and is bounded by $C_j (\log n)^j$ on such integers for a suitable constant $C_j$.

\begin{lem}\label{lem:BT}
Fix $\eps>0$ and $j\in\N$. Uniformly for $x\ge2$ and $x^{\eps} < y \le x$, we have
\[
\sum_{x < n \le x+y} \Lambda_j(n) \ll_{j,\eps} y (\log x)^{j-1}. 
\]
\end{lem}
\begin{proof} 	We argue by induction on $j$. The base case $j=1$ is a direct corollary of the classical Brun-Titchmarsh inequality (for example, see \cite[Theorem 3.9]{MV}).  Now suppose that $j\ge 2$ and that the lemma holds for $\Lambda_1,\dots, \Lambda_{j-1}$.  

The number of integers in $(x,x+y]$ all of whose prime factors are $\ge \sqrt{y}$ may be bounded by 
$\ll y/\log y \ll_{\eps} y/\log x$ (see \cite[Theorem 3.3]{MV}).  Therefore, with $P^{-}(n)$ denoting the least prime 
factor of the integer $n$, we have that
 \begin{equation} 
\label{N3.2} 
\sum_{\substack{ x< n\le x+y \\ P^{-}(n) >\sqrt{y}}} \Lambda_j(n) 
	\ll_j (\log x)^{j} \sum_{\substack{ x< n\le x+y \\ P^{-}(n) >\sqrt{y}}} 1
	\ll_{j,\eps} y (\log x)^{j-1}. 
\end{equation} 
To establish the lemma, it remains to show that 
\begin{equation} 
\label{N3.3} 
\sum_{\substack{ x< n\le x+y \\ P^{-}(n) \le \sqrt{y}}} \Lambda_j(n) \ll_{j,\eps} y (\log x)^{j-1}. 
\end{equation} 

Let $p$ be a prime and suppose $n= p^a m$ with $a\ge 1$ and $p \nmid m$.  Note that 
\begin{align*}
\Lambda_j(n) 
	&= \sum_{d|n} \mu(d) \log^j(n/d)
		 = \sum_{d|m} \mu(d) \Big( \log^j(p^a m/d) - \log^j(p^{a-1}m/d)\Big)\\
	&= \sum_{\ell=1}^{j} \binom{j}{\ell}  \big( \log^\ell(p^a) - \log^\ell(p^{a-1}) \big) 
		\sum_{d|m} \mu(d) \log^{j-\ell}(m/d) \\
	&=\sum_{\ell=1}^{j} \binom{j}{\ell} \big( \log^\ell(p^a) - \log^\ell(p^{a-1}) \big)
		  \Lambda_{j-\ell}(m). 
\end{align*}
If $a=1$, then we deduce that 
\begin{equation}\label{lambda-recurrence}
\Lambda_j(n) \ll_j (\log p)^j  1_{m=1} +(\log p)^{j-1} \Lambda_1(m) +\ldots +(\log p)\Lambda_{j-1}(m). 
\end{equation}
On the other hand, if $a>1$, then we use the bound $\Lambda_{j-\ell}(m)\ll_j (\log m)^{j-\ell}$ to conclude that
\begin{equation}\label{lambda-recurrence2}
\Lambda_j(n) \ll_j \log(p^a) (\log m)^{j-1}+\log^j(p^a) 
\ll \log(p^a)(\log n)^{j-1} .
\end{equation}

We now return to the task of estimating \eqref{N3.3}, using the above two estimates. 
Let $p$ denote the smallest prime factor of $n$, so that $p\le \sqrt{y}$. 
The terms with $p\Vert n$ contribute, using the induction hypothesis and \eqref{lambda-recurrence}, 
\begin{align} 
\label{N3.4}
&\ll_j \sum_{p\le \sqrt{y}} \,
	\sum_{x/p < m \le (x+y)/p}\big((\log p)^{j-1} \Lambda_1(m) +\ldots +(\log p)\Lambda_{j-1}(m)\big)
\nonumber \\
&\ll_j \sum_{p\le \sqrt{y}} 	
		\Big( (\log p)^{j-1}  \cdot  \frac{y}{p} +
		+\ldots+
		 (\log p)  \cdot  \frac{y}{p} (\log x)^{j-2} \Big) \nonumber\\
&\ll y (\log x)^{j-1}.
\end{align} 
Lastly, using \eqref{lambda-recurrence2}, we find that the terms with $p^2|n$ contribute 
\begin{align} 
\label{N3.5} 
&\ll (\log x)^{j-1}  \mathop{\sum\sum}_{\substack{p \le \sqrt{y} ,\, a \ge 2 \\ p^a \le x+y} } 
\log(p^a) \sum_{x/p^a \le m \le (x+y)/p^a} 1 \nonumber\\
& \ll 
(\log x)^{j-1}  \mathop{\sum\sum}_{\substack{p \le \sqrt{y} ,\, a \ge 2 \\ p^a \le x+y} } 
	 \log(p^a) \Big( \frac{y}{p^a} + 1 \Big)\ll y(\log x)^{j-1}.
\end{align} 
Combining \eqref{N3.4} and \eqref{N3.5} yields \eqref{N3.3}, completing the proof of the lemma. 
\end{proof}



\section{Proof of Proposition \ref{Prop1}} 

Recall that part (a) of Proposition \ref{Prop1} was already established in Lemma \ref{lem:L-der-bound}.  We now turn to the 
remaining three parts of the proposition, with the next lemma settling part (b).

\begin{lem} \label{lem:logL} Let $f\in {\mathcal F}(D;A)$ with $A> D+1$. For any real number $\gamma$, there exists 
an integer $j\in[0,D]$ with $L^{(j)}(1+i\gamma, f)\neq 0$. The \emph{multiplicity} of the zero of $L(s,f)$ at $s=1+i\gamma$ is defined as the smallest such $j$ with $L^{(j)}(1+i\gamma, f)\neq 0$.  If $m$ is the multiplicity 
of $1+i\gamma$ (we allow the possibility that $m =0$, which occurs when $L(1+i\gamma, f)\neq 0$), then 
\[
\Big| \sum_{p\le x} \frac{m+\Re (f(p)p^{-i\gamma})}{p}\Big| \le C
\] 
for some constant $C=C(f,\gamma)$.  
\end{lem} 

\begin{proof}  As $\sigma \to 1^+$, Taylor's theorem\footnote{The Dirichlet series $L(s,f)$ can be extended to function that is continuously differentiable function $D$ times in the half-plane $\mathbb{H}:=\{s\in\C:\Re(s)\ge1\}$, because the series $\sum_{n\ge1}f(n)(\log n)^j/n^s$ converges uniformly in $\mathbb{H}$ by our assumption that $f\in\CF(F;A)$ with $A>D+1$ and by partial summation. Note that the derivatives on the line $\Re(s)=1$ are defined as one-sided limits. Taylor's theorem with integral remainder term then implies that $L(\sigma+i\gamma,f)=\sum_{j=0}^{D-1} (\sigma-1)^j L^{(j)}(1+i\gamma,f)/j!+\int_1^\sigma (\sigma-u)^{D-1} L^{(D)}(u+i\gamma,f) /D! \dee u$. The last term of this expansion equals $(\sigma-1)^DL^{(D)}(1+i\gamma,f)+o((\sigma-1)^D)$ when $\sigma\to1^+$.}
shows that 
\[ 
L(\sigma+ i\gamma, f) = \sum_{j=0}^{D} \frac{(\sigma-1)^j}{j!} L^{(j)}(1+i\gamma, f) + o((\sigma -1)^D). 
\] 
But since $\Re(f(p)p^{-i\gamma}) \ge -D$ for all $p$, relation \eqref{2.2} implies that
\[
|L(\sigma+i\gamma,f)| \gg \exp\Big(\sum_{p}\frac{ \Re(f(p)p^{-i\gamma})}{p^{\sigma}}\Big) \gg (\sigma -1)^D. 
\] 
Therefore $L^{(j)}(1+i\gamma, f) \neq 0$ for some $0\le j\le D$, and the notion of multiplicity 
is well defined.  

If $m\le D$ denotes the multiplicity, then a new application of Taylor's theorem gives 
\[ 
L(\sigma +i\gamma, f) = \frac{(\sigma -1)^m}{m!} L^{(m)}(1+i\gamma, f) + o((\sigma -1)^m).
\] 
Writing $\sigma =1+1/\log x$ and taking logarithms, we find that
\[ 
\Big|\sum_{p \le x} \frac{m+\Re(f(p)p^{-i\gamma})}{p} \Big| =\Big|  \log |L(1+1/\log x+i\gamma, f)| + m \log \log x + O(1) \Big| 
\le C(f,\gamma),  
\] 
as desired.  
\end{proof} 

We now turn to the task of proving part (c) of Proposition \ref{Prop1}.  
Suppose $1+i\gamma_1,\dots,1+i\gamma_k$ are distinct zeros of $L(s,f)$, and let $m_j$ denote the multiplicity of the zero $1+i\gamma_j$.  
We wish to show that $m _1+ \cdots + m_k \le D$, so that part (c) would follow.   A key role will be played by the auxiliary function
\[
A_N(x) = \lim_{T\to \infty} \frac{1}{2T} \int_{-T}^{T} | 1 + e^{it\gamma_1} + \ldots +e^{i t \gamma_k}|^{2N} e^{it x} \dee t,
\] 
where $N$ is an integer that will be chosen large enough. By expanding the $(2N)$-th power, it is easy to see that $A_N(x)$ is non-zero only for those real $x$ that may be written as $j_1 \gamma_1 + \cdots + j_k \gamma_k$ with $|j_1| + \cdots + |j_k| \le N$.  Note that there may be linear relations among the $\gamma_j$, so that $A_N(x)$ could have a complicated structure. The following lemma summarizes the key properties of $A_N(x)$ for our purposes.

\begin{lem}\label{lem:A_N(x)} Let $N$ be a natural number.
\begin{enumerate}
	\item We have $A_N(0) \gg_k (k+1)^{2N}N^{-k/2}$.
	\item Let $\eps>0$ and $j\in\{1,\dots,k\}$. If $N$ is large enough in terms of $\eps$ and $k$, then
	$A_N(\gamma_j) \ge (1-\eps) A_N(0)$.
\end{enumerate}
\end{lem}

\begin{proof}
(a) Let $\gamma_0=0$. Then, the multinomial theorem implies that
\[
A_N(0) = \mathop{\sum\cdots\sum}_{\substack{j_0+j_1+\cdots+j_k=j_0'+j_1'+\cdots+j_k'=N \\ j_0\gamma_0+j_1\gamma_1+\cdots+j_k\gamma_k=j_0'\gamma+j_1'\gamma_1+\cdots+j_k'\gamma_k \\ j_i,j_i'\ge0\ \forall i}} \binom{N}{j_0,j_1,\dots,j_k}\binom{N}{j_0',j_1',\dots,j_k'} .
\] 
By restricting our attention on the `diagonal' terms with $j_i'=j_i$ for all $i$, we infer that
\[
A_N(0) \ge \mathop{\sum\cdots\sum}_{\substack{j_0+j_1+\cdots+j_k=N \\ j_i\ge0\ \forall i}}  \binom{N}{j_0,j_1,\dots,j_k}^2 .
\]
When $|j_i-N/k|\le \sqrt{N}$ for all $i$, Stirling's formula implies that the corresponding binomial coefficient has size $\asymp_k (k+1)^N N^{-k}$. Since there are $\asymp_k N^{k/2}$ tuples $(j_0,j_1,\dots,j_k)\in\Z_{\ge0}^{k+1}$ with $j_0+j_1+\cdots+j_k=N$, the claimed lower bound on $A_N(0)$ follows readily. 

\medskip

(b) Let $\CT$ denote the set of $t\in[-T,T]$ such that $\cos (t \gamma_j) \le 1-4\delta$. 
	Then $|1+e^{it\gamma_j}|\le2\sqrt{1-2\delta}\le2-2\delta$, whence $|1+e^{it \gamma_1} + \cdots +e^{it\gamma_k}| \le k+1- 2\delta$.  Therefore, in view of part (a), we have that
	\[
	\frac{1}{2T}\int_{\CT} 
	|1+e^{it\gamma_1}  + \cdots + e^{it\gamma_k}|^{2N} \dee t	
	\le (k+1-2\delta)^{2N} \le \delta A_N(0),
	\]
	provided that $N$ is large enough. Hence, if $T$ is sufficiently large,  
	\als{
		\frac{1}{2T}\int_{-T}^{T} \cos(t\gamma_j)
		|1+e^{it\gamma_1}  + \cdots + e^{it\gamma_k}|^{2N} \dee t
		& \ge 
		\frac{1-4\delta}{2T} \int_{[-T,T]\setminus \CT} |1+e^{it\gamma_1}  + \cdots + e^{it\gamma_k}|^{2N} \dee t \\
		&\qquad	
		- \frac{1}{2T}\int_{\CT} |1+e^{it\gamma_1}  + \cdots + e^{it\gamma_k}|^{2N} \dee t	 \\
		&\ge (1-4\delta)(1-\delta)A_N(0) - \delta A_N(0).
	}
 Taking $\delta$ suitably small in terms of $\eps$ completes the proof of the lemma.
\end{proof}

\begin{proof}[Proof of Proposition \ref{Prop1}(c)]
Let $N$ be  a large integer   to be chosen later, and consider the behavior of 
\begin{equation} \label{eqn:lambda-dfn}
\lambda_N(x):= \frac{1}{\log \log x} \Re \Big( \sum_{p\le x} \frac{f(p)}{p} |1+p^{i\gamma_1} + \cdots + p^{i\gamma_k}|^{2N}\Big),  
\end{equation} 
as $x\to \infty$.  We will estimate this quantity in two distinct ways, one of which will produce a lower bound and another one which will produce an upper bound. Comparing these bounds will then show that $m_1+\cdots+m_k\le D$.

For the lower bound on $\lambda_N(x)$, we note that our assumption that $|f(p)| \le D$ for all primes $p$ implies that
 \begin{equation} 
 \label{eqn:lambda-lb} 
\lambda_N(x) 
	\ge  \frac{ -D}{\log \log x}  \sum_{p\le x} \frac{1}{p} |1+p^{i\gamma_1} + \cdots + p^{i\gamma_k}|^{2N} 
	= - D A_N(0) +o(1),
\end{equation}
with the second relation following from the Prime Number Theorem.

Let us now bound $\lambda_N(x)$ from above. Expanding $|1+p^{-i\gamma_1} + \cdots +p^{-i\gamma_k}|^{2N}$, we find that
\[
\lambda_N(x) = \frac{1}{\log \log x}\sum_{\substack{0\le j_1,\dots,j_{2N}\le k \\ \gamma=\sum_{n\le N}\gamma_{j_n}-\sum_{n>N}\gamma_{j_n}}} 
	\sum_{p\le x} \frac{\Re(f(p)p^{-i\gamma})}{p} 
\]
with the convention that $\gamma_0=0$.  If now $\gamma=\gamma_\ell$ for some $\ell$, then the sum over $p$ equals $-m_\ell\log\log x+O(1)$.  
The number of choices of $j_1$, $\ldots$, $j_{2N}$ that lead to $\gamma= \gamma_\ell$ is exactly $A_N(\gamma_\ell)$.   
If $\gamma$ is not $\gamma_\ell$ for some $1\le \ell \le k$, then by Lemma \ref{lem:logL}  
we see that the sum over $p$ is bounded above by a constant.   Indeed if $\gamma$ is not an ordinate of a zero of $L(s,f)$ on the $1$--line, then 
the sum over $p$ is simply $O(1)$; {\sl a priori}, there could be other zeros of $L(s,f)$ besides $1+i\gamma_1$, $\ldots$, $1 +i\gamma_k$ and $\gamma$ could 
be one of these zeros, but 
nevertheless the sum over $p$ is bounded above by $O(1)$.    In conclusion,
\[
\lambda_N(x) \le - \sum_{\ell=1}^k m_\ell \sum_{\substack{0\le j_1,\dots,j_{2N}\le k \\ \sum_{n\le N}\gamma_{j_n}-\sum_{n>N}\gamma_{j_n} = \gamma_\ell}}1+o(1)
	= - \sum_{\ell=1}^k m_\ell A_N(\gamma_\ell)  +o(1).
\] 

Comparing the above inequality with \eqref{eqn:lambda-lb}, we infer that
\begin{equation} 
\label{eqn:multiplicities-ineq} 
\sum_{\ell=1}^{k} m_\ell A_N(\gamma_\ell) \le D A_N(0). 
\end{equation} 
To complete the proof, we apply Lemma \ref{lem:A_N(x)}(b) with $\eps=1/(m_1+\cdots+m_k+1)$ to find that the left-hand side of \eqref{eqn:multiplicities-ineq} is $>A_N(0)(m_1+\cdots+m_k-1)$, as long as $N$ is large enough. Since $A_N(0)>0$, we conclude that $m_1+\cdots+m_k<D+1$, as desired. \end{proof}

It remains lastly to prove part (d) of Proposition \ref{Prop1}. 
Suppose that the multiset $\widetilde \Gamma$ consists of $k$ distinct values, and has been arranged so that $\gamma_1,\dots, \gamma_k$ are 
these distinct values, and each such $\gamma_j$ occurs in $\widetilde \Gamma$ with multiplicity $\tilde{m}_j$.  
As in Lemma \ref{lem:DivisorSums}, put $\tau_{\widetilde \Gamma}(n) = \sum_{d_1 \cdots d_m = n} d_1^{i\gamma_1} \cdots d_m^{i\gamma_m}$ and define $f_{\widetilde \Gamma}$ to be the Dirichlet 
convolution $f * \tau_{\widetilde \Gamma}$.

\begin{lem}\label{lem:f-tau}  
With the above notations, we have 
\[
\sum_{n\le x} f_{\widetilde \Gamma}(n) \ll  C(f) \frac{x}{(\log x)^{A- \widetilde{m}}} + \frac{x (\log \log x)^{2D}}{(\log x)^A},  
\] 
for some constant $C(f)$, and with ${\widetilde m}$ denoting the maximum of the multiplicities $\tilde{m}_1,\dots,\tilde{m}_k$.
\end{lem}

\begin{proof}   As in the hyperbola method we may write, for some parameter $2\le z\le \sqrt{x}$ to be chosen shortly,  
\[ 
\sum_{n\le x} f_{\widetilde \Gamma}(n) = \sum_{a\le x/z} f(a) \sum_{b\le x/a} \tau_{\widetilde \Gamma}(b) + \sum_{b\le z} \tau_{\widetilde \Gamma}(b) \sum_{x/z \le a \le x/b} f(a). 
\] 
Using our hypothesis on the partial sums of $f$, and since $\sqrt{x} \le x/z \le x/b$, we see that the second term above 
is 
\begin{equation}
\label{eqn:f-tau-1} 
\ll \sum_{b\le z} |\tau_{\widetilde \Gamma}(b)| \frac{x}{b (\log x)^A} \ll \frac{x (\log z)^D}{(\log x)^A}, 
\end{equation}
since $|\tau_{\widetilde \Gamma}(b)|$ may be bounded by the $D$-th divisor function.   
On the other  hand, Lemma \ref{lem:DivisorSums} implies that there is some $\delta=\delta(\tilde{\Gamma})>0$ such that the first term equals 
\begin{equation} 
\label{eqn:f-tau-2} 
\sum_{a\le x/z} f(a) \sum_{j=1}^{k} \frac{x^{1+i\gamma_j}}{a^{1+i\gamma_j}} P_{j,\widetilde\Gamma}(\log x/a) + O\Big( x^{1-\delta} \sum_{a\le x/z} \frac{|f(a)|}{a^{1-\delta}} \Big), 
\end{equation} 
where $P_{j,\widetilde \Gamma}$ denotes a polynomial of degree $\widetilde{m}_j-1$ with coefficients depending on $f$ and $\widetilde \Gamma$.   
Since $|f(a)|$ is bounded by the $D$-th divisor function,
 the error term in \eqref{eqn:f-tau-2}  is easily bounded by $\ll x(\log x)^D/z^{\delta}$.  Now consider the main term in \eqref{eqn:f-tau-2}. Applying  \eqref{L-der-bound1} (with $N=x/z$ and $M\to \infty$ there), for any $0 \le \ell \le m_j-1$ we have 
\[
 \sum_{a \le x/z} \frac{f(a)}{a^{1+i\gamma_j}} (\log a)^\ell = (-1)^{\ell} L^{(\ell)}(1+i\gamma_j) + O_f\Big( \frac{1}{(\log x)^{A-\ell-1} }\Big) \ll_f \frac{1}{(\log x)^{A-\ell-1}}, 
\]
since $L^{(\ell)}(1+i\gamma_j) =0$ for all $0\le \ell \le m_j-1$.  
Therefore 
\[ 
\sum_{a\le x/z} \frac{f(a)}{a^{1+i\gamma_j}} P_{j,\widetilde \Gamma}(\log x/a) \ll \frac{1}{(\log x)^{A-m_j}},  
\]
and we conclude that the quantity in \eqref{eqn:f-tau-2} is 
\[
\ll \frac{x}{(\log x)^{A-\widetilde{m}}} + \frac{x(\log x)^D}{z^{\delta}}.
\]
Combine this with \eqref{eqn:f-tau-1}, and choose $z=\exp((\log \log x)^2)$ to obtain the lemma. 
\end{proof} 

Combining Lemma \ref{lem:f-tau} with the argument of Lemma \ref{lem:L-der-bound}, part (d) of Proposition \ref{Prop1} follows.

\smallskip


\section{Proof of Theorem \ref{converse-thm} in a special case}

In this section we establish Theorem \ref{converse-thm} in the special case when $L(s,f)$ has a 
zero of multiplicity $D$, say at $1+i\gamma$.  By Proposition \ref{Prop1} there can be no other zeros of $L(s,f)$ on the 
$1$-line.   In this special case, we can in fact prove a stronger result,  obtaining non-trivial information in the optimal range $A>D+1$.   
In the next section, we shall consider (by a very different method) the remaining cases when the multiplicity of any 
zero is at most $D-1$.  
 
 Write $g(n) = f(n) n^{-i\gamma}$, and consider $G= \tau_D *g$.   We begin by establishing some 
 estimates for $\sum_{n\le x} G(n)$ and $\sum_{n\le x} |G(n)|/n$.   Note that $G(n) = n^{-i\gamma} f_{\Gamma}$ for the 
 multiset $\Gamma$ composed of $D$ copies of $\gamma$.    Hence, Lemma \ref{lem:f-tau} and partial summation imply that
\eq{eq:g-bound1}{
\sum_{n\le x} G(n) \ll_f \frac{x}{(\log x)^{A-D}}.
}

 By Lemma \ref{lem:logL} we have 
 $$ 
\Big| \sum_{p\le x} \frac{\Re(G(p))}{p} \Big| = \Big| \sum_{p\le x} \frac{D+\Re(g(p))}{p} \Big| = \Big| \sum_{p\le x} \frac{D+ \Re(f(p)p^{-i\gamma})}{p} \Big| 
\ll_f 1. 
$$ 
 Since $|D+g(p)|^2 = D^2+ 2D\Re(g(p)) + |g(p)|^2 \le 2D (D+\Re(g(p)))$, an application of Cauchy-Schwarz gives 
 \begin{equation} 
 \label{N5.2} 
 \sum_{p\le x}\frac{|D+g(p)|}{p} \ll_f \sqrt{\log\log x}. 
 \end{equation} 
 It follows that 
 \begin{equation} 
 \label{N5.3} 
\sum_{n\le x}\frac{|G(n)|}{n} \ll \exp\Big( \sum_{p\le x} \frac{|G(p)|}{p} \Big)   \ll \exp\big(O_f\big(\sqrt{\log\log x}\,\big)\big).
\end{equation}
 
 After these preliminaries, we may now begin the proof of Theorem \ref{converse-thm} in this situation.  
 We shall consider the function $G*\overline{G} = \tau_{2D}*g*\overline{g}$.  Note that $\Lambda_{G*\overline{G}}(n) 
 = 2D \Lambda(n) + \Lambda_g(n) + \Lambda_{\overline{g}}(n)$ is always real and non-negative.  Thus $G*\overline{G}$ is 
 also a real and non-negative function, and we have 
 $$ 
 2\sum_{p\le x} (D+ \Re(g(p))) = \sum_{p\le x} (G*{\overline G})(p) \le \sum_{n\le x} (G*\overline{G})(n). 
 $$ 
 We bound the right side above using the hyperbola method.  Thus, using \eqref{eq:g-bound1} and \eqref{N5.3}, 
 \begin{align*}
\sum_{n\le x} (G*\bar{G})(n) 	
	&= 2\Re\Big(\sum_{a\le \sqrt{x}} G(a) \sum_{b\le x/a} \bar{G}(b) \Big)- \Big|\sum_{a\le \sqrt{x}}G(a)\Big|^2 \\ 
	&\ll_f \frac{x}{(\log x)^{A-D}} \sum_{a\le x} \frac{|G(a)|}{a} + \frac{x}{(\log x)^{2(A-D)}} \ll_{f,\eps} \frac{x}{(\log x)^{A-D-\eps}}
\end{align*}
for any fixed $\eps>0$. Thus 
\[
\sum_{p\le x} \big(D+ \Re(g(p)) \big)
	 \ll_{f,\eps} \frac{x}{(\log x)^{A-D-\eps}}, 
\]
 and using $|D+g(p)|^2 \le 2D (D+\Re(g(p)))$ and Cauchy-Schwarz we conclude that 
 \begin{equation} 
 \label{N5.4} 
 \sum_{p\le x} |D+ f(p)p^{-i\gamma}| \log p = \sum_{p\le x} |f(p) + Dp^{i\gamma}| \log p 
			 \ll_{f,\eps} \frac{x}{(\log x)^{(A-1-D-\eps)/2}}.
 \end{equation}

 Once the estimate \eqref{N5.4} has been established, it may be input into the above argument and the bound \eqref{N5.4} 
 may be tidied up.  Partial summation starting from \eqref{N5.4} leads to the bound $\sum_{p\le x} |D+g(p)|/p \ll_f 1$ in place of 
 \eqref{N5.2}.  In turn this replaces \eqref{N5.3} by the bound $\sum_{n\le x} |G(n)|/n \ll_f 1$.  Using this in our hyperbola 
 method argument produces now the cleaner bound 
 \begin{equation} 
 \label{N5.5} 
  \sum_{p\le x} |D+ f(p)p^{-i\gamma}| \log p = \sum_{p\le x} |f(p) + Dp^{i\gamma}| \log p \ll_f \frac{x}{(\log x)^{(A-1-D)/2}}.
 \end{equation}

 As mentioned earlier, the estimate \eqref{N5.5} obtains non-trivial information in the optimal range $A>D+1$.  
 If we suppose that $A> D+2$, then the right side of \eqref{N5.5} is $\ll_f x/\sqrt{\log x}$, and Theorem \ref{converse-thm} 
 follows in this special case if $|\gamma|\le T$.  If $|\gamma| >T$ then note that 
 $$ 
 \sum_{p\le x} p^{i\gamma} \log p = \frac{x}{1+i\gamma} + O_f\Big( \frac{x}{\log x}\Big) \ll \frac{x}{T} + C(f) \frac{x}{\log x}, 
$$ 
so that the theorem holds as stated in this case also.


\section{Proof of  Theorem \ref{converse-thm}:  The general case} 
 
In the previous section we established Theorem \ref{converse-thm} in the special situation when $L(s,f)$ 
 has a zero of multiplicity $D$ on the $1$-line.  We now consider the more typical situation when all the zeros (if there are any)
 of $L(s,f)$ on the line $\Re(s)=1$ have multiplicity $\le D-1$.   The argument here is based on some ideas from \cite{GHS}.

 Throughout we put $c=1+1/\log x$, and $T_0 =\sqrt{T}$.  Let ${\widetilde \Gamma}$ denote 
the multiset of zeros of $L(s,f)$ lying on the line segment $[1-iT, 1+iT]$, and let $f_{\widetilde \Gamma} = f * \tau_{\widetilde \Gamma}$ denote 
the multiplicative function defined for Lemma \ref{lem:f-tau}. We start with a smoothed version of Perron's formula: 
\begin{align} 
\label{4.1} 
\frac{1}{2\pi i} \int_{(c)} \bigg(-\frac{L^{\prime}}{L}\bigg)^{\prime}(s,f_{\widetilde \Gamma}) \frac{x^s}{s}
	 \bigg(\frac{e^{s/T_0}-1}{s/T_0}\bigg)^{10} \dee s 
	 &= \sum_{n\le x} \Lambda_{f_{\widetilde \Gamma}}(n) \log n 
		+ O\Big( \sum_{x < n < e^{10/T_0} x} \Lambda(n) \log n\Big) \nonumber\\ 
	&= \sum_{p\le x} \Big(f(p) + \sum_{\substack{\gamma \in \Gamma \\ |\gamma| \le T}} p^{i\gamma}\Big) (\log p)^2  
		+ O\Big( \frac{x\log x}{T_0}\Big). 
\end{align}  
Our goal now is to bound the left-hand side of \eqref{4.1}, and to do this we split the integral into several ranges.   There is a range of 
small values $|t| \le T$, and the range of larger values $|t| >T$, 
which we further subdivide into dyadic ranges $2^rT  < |t| \le 2^{r+1} T$ with $r\ge 0$.  

\subsection{Small values of $|t|$}  

We start with the range $|t|\le T$. Since $A>D+2$ and all zeros of $L(s,f)$ are assumed to have multiplicity $\le D-1$, we have $A-{\widetilde m} -1 >2$.  Therefore, by Proposition \ref{Prop1}(d)  $L^{(j)}(s,f_{\widetilde \Gamma})$ exists for $\Re(s)\ge 1$ and $j\in\{0,1,2\}$, and is  bounded above in magnitude on the segment $[1-iT,1+iT]$.  
Further, $|L(s,f_{\widetilde \Gamma})|$ is bounded away from zero on the compact set $[1-iT, 1+iT]$ since all the zeros of $L(s,f)$ in that region 
are accounted for in the multiset $\widetilde \Gamma$.   Therefore there is some constant $C(f,T)$ such that for all $|t|\le T$ one has 
\[ 
\bigg| \bigg(\frac{L^{\prime}}{L}\bigg)^{\prime}(c+it,f_{\widetilde \Gamma}) \bigg| \le C(f,T). 
\]
We deduce that 
\begin{equation} 
\label{eqn:perron-bound1} 
\bigg| \int_{\sigma=c,\, |t|\le T}\bigg(\frac{L^{\prime}}{L}\bigg)^{\prime}(s,f_{\widetilde \Gamma}) 
	\cdot \frac{x^s}{s} \cdot 
		\bigg( \frac{e^{s/T_0}-1}{s/T_0}\bigg)^{10} \dee t \bigg|
				 \le x C_1(f,T), 
\end{equation}  
for a suitable constant $C_1(f,T)$.  

\subsection{Large values of $|t|$}    Now we turn to the larger values of $|t|$, namely when $2^r T < |t| \le 2^{r+1} T$ for some $r\ge 0$.  Writing 
\[
\bigg( \frac{L^{\prime}}{L}\bigg)^{\prime}(s,f_{\widetilde \Gamma}) 
		= \bigg( \frac{L^{\prime}}{L} \bigg)^{\prime}(s,f) 
		+ \sum_{\gamma \in \widetilde\Gamma} \bigg( \frac{\zeta^{\prime}}{\zeta} \bigg)^{\prime}(s-i\gamma) ,  
\] 
the desired integral splits naturally into two parts.   Now for $|t| \ge T$ we have 
\[ 
\bigg( \frac{\zeta^{\prime}}{\zeta} \bigg)'(c+it -i\gamma) 
		\ll \Big( \frac{1}{(\log x)^2} + |t-\gamma|^2 \Big)^{-1} + |t|^{\eps}, 
\] 
so that 
\begin{equation}
\label{4.12}
\bigg| \int\limits_{\substack{\sigma=c \\ 2^{r}T \le |t| \le 2^{r+1}T}} 
	\sum_{\gamma\in \widetilde\Gamma}
	\bigg( \frac{\zeta^{\prime}}{\zeta} \bigg)^{\prime}(s-i\gamma)
	\cdot 	\frac{x^s}{s} \cdot 
		\bigg( \frac{e^{s/T_0}-1}{s/T_0} \bigg)^{10} \dee s \bigg| 
\ll \frac{x\log x}{2^r T_0}.
\end{equation}

It remains now to estimate 
\begin{equation} 
\label{4.13} 
\bigg| \int\limits_{\substack{\sigma=c \\ 2^{r}T \le |t| \le 2^{r+1}T}}  
	\bigg(\frac{L^{\prime}}{L}\bigg)^{\prime}(s,f)
		\cdot \frac{x^s}{s}
		\cdot \bigg(\frac{e^{s/T_0}-1}{s/T_0}\bigg)^{10} \dee s \bigg|. 
\end{equation} 
To help estimate this quantity, we state the following lemma whose proof we postpone to 
the next section.

\begin{lem} 
\label{lem2.5}  Let $X \ge 2$ and $\sigma >1$ be real numbers. Let $f\in {\mathcal F}(D)$ and suppose $j\ge 1$ is a natural number.  
Put $G_j(s) = (-1)^j L^{(j)}(s,f)/L(s,f)$.   Then 
$$ 
\int_{-X}^X |G_j(\sigma +it)|^2 \dee t 
\ll_j X (\log X)^{2j} + \Big( \frac{1}{\sigma -1} \Big)^{2j-1}. 
$$ 
\end{lem} 

Returning to \eqref{4.13},  in the notation of Lemma \ref{lem2.5}, we have 
\[ 
\bigg(\frac{L^{\prime}}{L}\bigg)^{\prime}(s,f)
	 = G_2(s) - G_1(s)^2. 
\]   
Using this identity, the integral in \eqref{4.13} splits into two parts, and using Lemma \ref{lem2.5} we may bound the 
second integral (with $X=2^{r}T$) by 
\begin{align}
\label{4.2} 
 \frac{x}{X (X/T_0)^{10}} \int_{X< |t| \le 2X} |G_1(c+it)|^2 \dee t 
 	& \ll \frac{x}{X(X/T_0)^{10}} \Big( X (\log X)^2 + \log x \Big)
		\ll \frac{x \log x}{2^r T_0}. 
 \end{align} 
 
Finally, we must bound the integral arising from $G_2(s)$.  To this end, we define
\begin{equation} 
\label{4.3} 
I(j;X,\alpha) = \Big| \int\limits_{\substack{\sigma=c \\ X \le |t| \le 2X}}  
	 G_j(s+\alpha) \frac{x^s}{s} \Big( \frac{ e^{s/T_0}-1}{s/T_0} \Big)^{10} \dee s \Big|, 
\end{equation} 
so that we require a bound for $I(2;2^rT,0)$.  We shall bound $I(j;X,\alpha)$ in terms of $I(j+1;X, \alpha+\beta)$ for suitable 
$\beta>0$, and iterating this will eventually lead to a good bound for $I(2;X,0)$.  

 \begin{lem}\label{lem4.1}   Let $X\ge T$, and $\alpha\ge 0$ be real numbers.   For $j\ge 2$ and all $k\ge 1$ we have 
\[
  I(j;X,\alpha) \ll_k \int_0^1  I(j+k;X,\alpha+\beta)\beta^{k-1} \dee\beta + \frac{x}{(X/T_0)^2}  
  \Big(\frac{1}{\log x} +\alpha\Big)^{-(j-1)} . 
 \]
 \end{lem} 

\begin{proof} 
 Note that 
\[
  G_{j}(s)^{\prime} = - G_{j+1} (s) + G_1(s) G_j(s), 
\]
 so that 
  \begin{align*}
  G_j(s+\alpha) & = - \int_{0}^{1} G_{j}^{\prime}(s+\alpha+\beta) \dee\beta + O(1) \\
  &= \int_0^1 (G_{j+1}(s+\alpha+\beta) -G_1(s+\alpha+\beta) G_j(s+\alpha+\beta)) \dee\beta + O(1). 
 \end{align*} 
 Using this in the definition of $I$, we obtain that 
  \begin{align*} 
  I(j;X,\alpha) &\ll \int_{0}^{1} I(j+1;X,\alpha+\beta) \dee\beta \\
  &+   \frac{x}{X(X/T_0)^{10}} \Big( X + \int_0^1 
  \int_{X}^{2X}  |G_1(c+\alpha+\beta+ it) G_j(c+\alpha+\beta+it)| \dee t \dee\beta\Big). 
  \end{align*} 
  
Using Cauchy--Schwarz and Lemma \ref{lem2.5}, the second term above is (since $j\ge 2$) 
  \begin{align*}
&\ll \frac{x}{X(X/T_0)^{10}} \Big( X + X (\log X)^{j+1} + \int_0^{1} \Big(\frac{1}{\log x} + \alpha +\beta\Big)^{-j} \dee\beta \Big)   \\
&\ll \frac{x}{(X/T_0)^2} \Big( \frac{1}{\log x} + \alpha\Big)^{-(j-1)}. 
 \end{align*} 
Thus we conclude that 
\begin{equation} 
\label{4.4}  
I(j;X,\alpha)  \ll \int_{0}^{1} I(j+1;X,\alpha+\beta) \dee\beta + \frac{x}{(X/T_0)^2} \Big( \frac{1}{\log x} + \alpha\Big)^{-(j-1)}.
\end{equation} 

This establishes the lemma in the case $k=0$, and the general case  follows by iterating this argument $k-1$ times. In doing so, we make use of the following estimate:
\begin{equation}\label{eqn:int-bound}
\int_0^1 \frac{\beta^{m-1}}{(\alpha+\beta+1/\log x)^{j+m-1}}\dee\beta \ll_{m,j} \frac{1}{(\alpha+1/\log x)^{j-1}}
\end{equation}
for all $m=1,2,\dots$ and all $\alpha\in[0,1]$.  This may be seen by dividing the range of integration into two parts, according to whether $\beta\le \alpha+1/\log x$ or $\beta>\alpha+1/\log x$. \end{proof} 

We now return to the task of bounding $I(2;2^r T, 0)$.  Applying Lemma \ref{lem4.1}, we see that for any $k\ge 1$  we have
\begin{equation}
\label{4.5} 
I(2;2^r T,0) \ll_k \int_0^1 I(2+k;2^rT, \beta) \beta^{k-1} \dee\beta + \frac{x \log x}{2^r T_0}.
\end{equation} 
We choose $k$ to be the largest integer strictly smaller than $A-3$. Since $A>D+2$, we have
\[
D-1\le k<A-3.
\]
Applying Lemma \ref{lem:L-der-bound}, we find that $L^{(2+k)}(c+\beta+it,f) \ll (1+|t|)$.  
Furthermore, since $f \in {\mathcal F}(D)$, we have
\[
\frac{1}{|L(c+\beta+it,f)|} \ll \prod_{p} \Big(1 -\frac{1}{p^{c+\beta}}\Big)^{-D} \ll \Big( \frac{1}{\log x} + \beta\Big)^{-D}. 
\] 
Thus, with this choice of $k$, it follows that 
\[
I(2+k;2^r T, \beta) \ll x \Big( \frac{1}{\log x} + \beta\Big)^{-D}   \int_{2^r T}^{2^{r+1}T} \frac{1}{(T/T_0)^{10}} \dee t
	 \ll \frac{x}{2^r T_0} \Big( \frac{1}{\log x} +\beta\Big)^{-D} .
\]
Since $k\ge D-1$, we infer that
\[
\int_0^1 I(2+k;2^r T, \beta) \beta^{k-1} \dee \beta   \ll \frac{x\log x}{2^r T_0},
\]
by  a similar argument to the one leading to \eqref{eqn:int-bound}. 
In conclusion,
\[
I(2;2^{r} T, 0 ) \ll \frac{x\log x}{2^r T_0}. 
\] 
Combining this with \eqref{4.2}, and summing over all $r\ge 0$, we obtain 
\begin{equation} 
\label{4.6} 
\bigg| \int_{\sigma=c,\, |t|>T} \bigg(\frac{L^{\prime}}{L}\bigg)'(s,f) \cdot \frac{x^s}{s} \cdot \bigg( \frac{e^{s/T_0} -1}{s/T_0} \bigg)^{10} 
	\dee s \bigg| \ll \frac{x\log x}{T_0}. 
\end{equation} 

Combining \eqref{4.6} with \eqref{4.12} summed over all $r$, we conclude that 
\begin{equation} 
\label{4.10}
\bigg| \int_{\sigma=c,\, |t|>T} \bigg(\frac{L^{\prime}}{L}\bigg)'(s,f_{\widetilde \Gamma}) \cdot \frac{x^s}{s} \cdot 
\bigg( \frac{e^{s/T_0} -1}{s/T_0} \bigg)^{10} 
\dee s \bigg| \ll \frac{x\log x}{T_0}. 
\end{equation} 

\subsection{Completing the proof.}  Combining \eqref{4.10} with \eqref{4.1} and \eqref{eqn:perron-bound1}, it follows that 
$$ 
\sum_{p\le x} \Big( f(p) + \sum_{\substack{ \gamma \in \Gamma \\ |\gamma| \le T}} p^{i\gamma} \Big) (\log p)^2 \le xC_1(f,T) + O\Big( \frac{x\log x}{T_0} \Big).   
$$ 
Partial summation now finishes the proof of Theorem \ref{converse-thm}. 

%
%
%
%

\section{Proof of Lemma \ref{lem2.5}} 


Write $g_j(n)$ for the Dirichlet series coefficients of $G_j(n)$.  
We claim that $|g_j(n)| \le D^j \Lambda_j(n)$ for all $n$.  For $j=1$ this is just the definition of the class ${\mathcal F}(D)$.  
To see the claim in general, we use induction on $j$, noting that 
\begin{equation} 
\label{2.5} 
G_{j+1}(s) = - G_j'(s) + G_1(s) G_j(s), 
\end{equation} 
and now comparing this with \eqref{N3.1}.  Using this bound for $|g_j(n)|$, we have 
\[
\bigg| \sum_{n\le X^2} \frac{g_j(n)}{n^{c+it}}\bigg|\le \sum_{n\le X^2} \frac{D^j\Lambda_j(n)}{n}  \ll_j (\log X)^j, 
\]
so that 
\[ 
\int_{-X}^{X} \bigg| \sum_{n\le X^2} \frac{g_j(n)}{n^{c+it}}\bigg|^2 \dee t  \ll_j X (\log X)^{2j}.
\]

Next, putting $\Phi(x) = (\frac{\sin x}{x})^2$ so that the Fourier transform ${\widehat \Phi}(x)$ is supported on $[-1,1]$, 
we obtain 
\begin{align*}
\int_{-X}^{X} \bigg| \sum_{n > X^2} \frac{g_j(n)}{n^{c+it}}\bigg|^2 \dee t 
	&\ll \int_{-\infty}^{\infty} \Big| \sum_{n>X^2} \frac{g_j(n)}{n^{c+it}}\Big|^2 \Phi\Big( \frac{x}{X} \Big) \dee x  \\
	&\ll \sum_{m,n>X^2} \frac{\Lambda_j(m)\Lambda_j(n)}{(mn)^c} X \widehat{\Phi}(X \log (m/n)). 
\end{align*}
Since ${\widehat \Phi}$ is supported on $[-1,1]$ for a given $m >X^2$, the sum over $n$ is restricted to the range $|m-n| \ll m/X$, and so, using the variant of the Brun-Titchmarsh theorem Lemma \ref{lem:BT}, we deduce that the above is 
\[ 
\ll X \sum_{m >X^2} \frac{\Lambda_j(m)}{m^{2c}} \sum_{|m-n|\ll m/X} \Lambda_j(n) 
\ll_j \sum_{m >X^2} \frac{\Lambda_j(m)}{m^{2c}} m (\log m)^{j-1} 
\ll_j \Big( \frac{1}{c-1} \Big)^{2j-1}. 
\]
Lemma \ref{lem2.5} follows upon combining these two estimates.

\bibliographystyle{alpha}

\begin{dajauthors}
	\begin{authorinfo}[dimitris]
		Dimitris Koukoulopoulos\\
		Universit\'e de Montr\'eal\\
		CP 6128 succ. Centre-Ville\\
		Montr\'eal, QC H3C 3J7\\
		Canada\\
		koukoulo\imageat{}dms\imagedot{}umontreal\imagedot{}ca\\
	\url{https://www.dms.umontreal.ca/~koukoulo}
	\end{authorinfo}
	\begin{authorinfo}[sound]
		K. Soundararajan\\
		Stanford University \\
		450 Serra Mall, Building 380  \\
		Stanford, CA 94305-2125 \\
		USA\\
		ksound\imageat{}stanford\imageat{}edu\\
		\url{http://math.stanford.edu/~ksound/}
	\end{authorinfo}
\end{dajauthors}

\end{document}